\documentclass[12pt,reqno]{article}
\usepackage{amssymb,amscd,amsmath,amsthm,color}\usepackage{hyperref}
\newtheorem{theorem}{Theorem}[section]
\newtheorem{lemma}[theorem]{Lemma}
\newtheorem{cor}[theorem]{Corollary}
\newtheorem{proposition}[theorem]{Proposition}
\usepackage{enumerate,amssymb}

\newtheorem{conj}[theorem]{Conjecture}
\theoremstyle{remark}
\newtheorem{remark}[theorem]{Remark}

\numberwithin{equation}{section}

\begin{document}
\baselineskip=15pt

\title{On an operator Kantorovich inequality for positive linear maps}

\author{Minghua Lin}

\date{}

\maketitle

\begin{abstract}
\noindent We improve the operator Kantorovich inequality as follows: Let $A$ be a positive operator on a Hilbert space with $0<m\le A \le M$. Then for every unital positive  linear map $\Phi$, \[\Phi(A^{-1})^2\le \left(\frac{(M+m)^2}{4Mm}\right)^2\Phi(A)^{-2}.\]
As a consequence, \[\Phi(A^{-1})\Phi(A)+\Phi(A)\Phi(A^{-1}) \le \frac{(M+m)^2}{2Mm}.\]

\end{abstract}

{\small\noindent

\noindent
AMS subjects classification 2010: 47A63, 47A30.}

\section{Introduction}
 As customary, we reserve $M, m$ for scalars and $I$ for the identity operator. Other capital letters denote general elements of the $C^*$ algebra $\mathcal {B}(\mathcal {H})$ (with unit) of all bounded linear operator acting on a Hilbert space $(\mathcal {H}, \langle \cdot, \cdot\rangle)$. Also, we identify a scalar with the unit multiplied by this scalar.  The operator norm is denoted by $\|\cdot\|$.  We write  $T\ge 0$ to mean that the operator $T$ is positive and identify $T\ge S$ (the same as $S\le T$) with $T-S\ge 0$. A positive invertible operator $T$ is naturally denoted by $T>0$. Finally, the absolute value of $T$ is denoted by $|T|=(T^*T)^{1/2}$, where $T^*$ stands for the adjoint of $T$.

The Kantorovich inequality \cite{Kan48} (see also \cite[p. 444]{HJ85}) states that
for every unit vector $x\in \mathcal {H}$,  \begin{eqnarray}\label{e1}  \langle x, Ax\rangle\langle x, A^{-1}x\rangle \le  \frac{(M+m)^2}{4Mm}.
\end{eqnarray}

Operator version of (\ref{e1}) was firstly established by Marshall \& Olkin \cite{MO90}, who obtained

\begin{theorem}\label{thm1}\cite{MO90} Let  $0<m \le A \le M$.  Then for every  positive unital linear map $\Phi$, \begin{eqnarray}\label{op-kantorovich1}  \Phi(A^{-1})\le  \frac{(M+m)^2}{4Mm} \Phi(A)^{-1}.
\end{eqnarray}
  \end{theorem}

The operator Kantorovich inequality (\ref{op-kantorovich1}) can be regarded as a counterpart to Choi's inequality (see e.g. \cite[p. 41]{Bha07}), which says \begin{eqnarray}\label{choi}  \Phi(A^{-1})\ge  \Phi(A)^{-1},
\end{eqnarray} for every  positive unital linear map $\Phi$ and $A>0$.

There is considerable amount of literature devoting to the study of Kantorovich inequality, we refer to \cite{Mos12} for a recent survey and references therein. This paper intends to give an improvement of operator Kantorovich inequality (\ref{op-kantorovich1}) based on the following consideration:

~~

\noindent{\it It is well known that $t^{s}$ $(0\leq s\leq 1)$ is an operator monotone function  and not so is $t^2$.}

~~

 The main result is given in the next section. In the last section, we present some discussion and a conjecture.

\section{Main results}
For $A, B>0$, the geometric mean $A\sharp B$ is defined by \[A\sharp B=A^{1/2}(A^{-1/2}BA^{-1/2})^{1/2}A^{1/2}.\] Observing that the geometric mean is the unique positive solution to the Ricatti equation $XA^{-1}X=B$, we have $A\sharp B=B\sharp A$. One motivation of such a notion is of course the AM-GM inequality  \begin{eqnarray}\label{am-gm}  \frac{A+B}{2}\ge A\sharp B.\end{eqnarray} A remarkable property of geometric mean is the so called maximal characterization \cite{PW75}, which says that $\begin{bmatrix} A & A\sharp B \\ A\sharp B  &  B\end{bmatrix}$ is positive, and moreover, if the operator matrix $\begin{bmatrix} A &  X  \\ X  &  B\end{bmatrix}$, with $X$ being self-adjoint, is positive, then $A\sharp B\ge X$. Here $2\times 2$ operator matrix is naturally understood as an operator acting on $\mathcal {H}\oplus \mathcal {H}$.
It is well known that geometric mean is monotone, i.e., if $A\ge C>0$ and   $B\ge D>0$, then  $A\sharp B\ge C\sharp D$. The next property may be known as well.

\begin{proposition}\label{prop0} If $0< A\le B$, then $A\sharp B^{-1}\le I$. The converse is not true.  \end{proposition}
\begin{proof}  Operator reverse monotonicity of the inverse tells us $A^{-1}\ge B^{-1}>0$. Then $A\sharp B^{-1}\le A\sharp A^{-1}=I$. To see why the converse fails, suppose  $A\sharp B^{-1}\le I$, then  $A^2\sharp B^{-2}\le I$ (see \cite{And87}). If the converse is true, then $0< A\le B$ would imply $A^2\le B^2$, which is absurd.
\end{proof}

Thus the operator Kantorovich inequality (\ref{op-kantorovich1}) implies
\begin{cor}\label{g-kantorovich} Let  $0<m \le A \le M$.  Then for every  positive unital linear map $\Phi$, \begin{eqnarray}\label{op-kantorovich2}  \Phi(A^{-1})\sharp \Phi(A)\le  \frac{ M+m }{2\sqrt{Mm}}.
\end{eqnarray}
  \end{cor}
\begin{remark} As pointed out by a referee, a more general case of (\ref{g-kantorovich}) has been studied in the context of  matrix reverse Cauchy-Schwarz inequality; see \cite[Theorem 4]{Lee09}.  \end{remark}

There is another way to approach (\ref{op-kantorovich2}). Note that \[(M-A)(m-A)A^{-1}\le 0,\] then \[Mm A^{-1}+A\le M+m,\]
and hence \begin{eqnarray}\label{e3}  Mm \Phi(A^{-1})+\Phi(A)\le M+m. \end{eqnarray}
Applying the AM-GM inequality  (\ref{am-gm}) to the left hand side of (\ref{e3}), we obtain (\ref{op-kantorovich2}).

Fujii et al. \cite{FINS97} proved that $t^2$ is order preserving in the following sense.
\begin{proposition}\label{prop1} Let $0\le A\le B$ and  $0<m \le A \le M$. Then \[A^2\le \frac{(M+m)^2}{4Mm}B^2.\]
\end{proposition}

A quick use of the Proposition \ref{prop1} gives the following preliminary result.
\begin{proposition} Let  $0<m \le A \le M$. Then  for every  positive unital linear map $\Phi$,  \begin{eqnarray}\label{pre}  \Phi(A^{-1})^2\le  \left(\frac{(M+m)^2}{4Mm}\right)^3 \Phi(A)^{-2}.
\end{eqnarray}
\end{proposition}
\begin{proof} As $\Phi$ is order preserving, then \[\frac{1}{M}\le \Phi(A^{-1})\le \frac{1}{m}.\] By Proposition \ref{prop1}, we have
\begin{eqnarray*}  \Phi(A^{-1})^2&\le& \frac{\left(\frac{1}{M}+\frac{1}{m}\right)^2}{4\frac{1}{M}\frac{1}{m}}\left(\frac{(M+m)^2}{4Mm}\Phi(A)^{-1}\right)^2\\&=& \left(\frac{(M+m)^2}{4Mm}\right)^3 \Phi(A)^{-2}.
\end{eqnarray*}
\end{proof}

As indicated in the Abstract, we are not content with the factor $\left(\frac{(M+m)^2}{4Mm}\right)^3$ in the previous proposition.  The ideal factor should be $\left(\frac{(M+m)^2}{4Mm}\right)^2$. We need a lemma to establish this fact.

\begin{lemma}\label{lem1}\cite{BK00}  Let $A, B\ge 0$.  Then the following norm inequality is valid
\begin{eqnarray}\label{e4} \|AB\|\le \frac{1}{4}\|A+B\|^2.
\end{eqnarray} \end{lemma}

\begin{remark}  Drury \cite{Dru12} recently established a remarkable improvement of (\ref{e4}) when $A, B$ are moreover compact. More precisely, if $A, B\ge0$ are compact, then there exists an isometry $U$ such that \[U|AB|U^*\le \frac{1}{4}(A+B)^2.\]  \end{remark}

Now we are going to present the main theorem of this paper.

\begin{theorem}  Let $0<m\le A \le M$. Then for every positive unital linear map $\Phi$, \begin{eqnarray}\label{op-kantorovich3a}  \|\Phi(A^{-1}) \Phi(A)\| \le \frac{(M+m)^2}{4Mm},\end{eqnarray} or equivalently, \begin{eqnarray}\label{op-kantorovich3b}  \Phi(A^{-1})^2\le\left(\frac{(M+m)^2}{4Mm}\right)^2\Phi(A)^{-2}.\end{eqnarray}
  \end{theorem}
\begin{proof}  By Lemma \ref{lem1}, we have  \begin{eqnarray*}     Mm\|\Phi(A^{-1}) \Phi(A)\|&=&\|Mm\Phi(A^{-1}) \Phi(A)\|\\ &\le& \frac{1}{4}\|Mm\Phi(A^{-1})+\Phi(A)\|^2\\&\le &\frac{1}{4}\|(M+m)I\|^2=\frac{(M+m)^2}{4},
\end{eqnarray*} where the second inequality is by (\ref{e3}). Therefore,  $\|\Phi(A^{-1}) \Phi(A)\| \le \frac{(M+m)^2}{4Mm}$.

As every positive linear map is adjoint-preserving, i.e., $\Phi(T^*)=\Phi(T)^*$ for all $T$, to complete the proof,  note that  (\ref{op-kantorovich3a}) is equivalent to  \[ \Phi(A)\Phi(A^{-1})^2\Phi(A)\le \left(\frac{(M+m)^2}{4Mm}\right)^2,\]  and hence  \[ \Phi(A^{-1})^2\le \left(\frac{(M+m)^2}{4Mm}\right)^2\Phi(A)^{-2}.\]
\end{proof}

Further refinement of (\ref{op-kantorovich3a}) might not be true, for example, it is easy to find counterexamples that \[\|\Phi(A^{-1})\|\|\Phi(A)\|\le\frac{(M+m)^2}{4Mm}\] fails.

The next observation is useful in our derivation of Theorem \ref{thm2} below.
\begin{lemma}\label{lem2}
   For any bounded operator $X$,
\begin{eqnarray}\label{useful} |X|\le tI\Leftrightarrow  \|X\|\le t\Leftrightarrow \begin{bmatrix} tI &  X  \\ X^*  &  tI\end{bmatrix}\ge0.
\end{eqnarray}  \end{lemma}
\begin{proof} It can be found, for example, in \cite[Lemma 3.5.12]{HJ91}.  \end{proof}

 \begin{theorem} \label{thm2} Let  $0<m \le A \le M$.  Then for every  positive unital linear map $\Phi$, \begin{eqnarray}\label{op-kantorovich5a}  |\Phi(A^{-1})\Phi(A)+\Phi(A)\Phi(A^{-1})| \le \frac{(M+m)^2}{2Mm}
\end{eqnarray} and \begin{eqnarray}\label{op-kantorovich5b}  \Phi(A^{-1})\Phi(A)+\Phi(A)\Phi(A^{-1}) \le \frac{(M+m)^2}{2Mm}.
\end{eqnarray}
  \end{theorem}
\begin{proof}  By (\ref{op-kantorovich3a}) and Lemma \ref{lem2}, we have \begin{eqnarray*}   \begin{bmatrix}  \frac{(M+m)^2}{4Mm}I &  \Phi(A^{-1}) \Phi(A)  \\  \Phi(A)\Phi(A^{-1})   &  \frac{(M+m)^2}{4Mm}I\end{bmatrix}\ge 0 ~~\hbox{and}~~  \begin{bmatrix}  \frac{(M+m)^2}{4Mm}I & \Phi(A)\Phi(A^{-1})  \\   \Phi(A^{-1}) \Phi(A)   &  \frac{(M+m)^2}{4Mm}I\end{bmatrix}\ge 0.
\end{eqnarray*}
Summing up these two operator matrices, we have   \begin{eqnarray*}   \begin{bmatrix}  \frac{(M+m)^2}{2Mm}I &  \Phi(A^{-1}) \Phi(A)+\Phi(A)\Phi(A^{-1})   \\ \Phi(A^{-1}) \Phi(A)+\Phi(A)\Phi(A^{-1})    &  \frac{(M+m)^2}{2Mm}I\end{bmatrix}\ge 0.
\end{eqnarray*} By Lemma \ref{lem1} again,  (\ref{op-kantorovich5a}) follows.  As $\Phi(A^{-1}) \Phi(A)+\Phi(A)\Phi(A^{-1})$ is self-adjoint, (\ref{op-kantorovich5b}) follows from the maximal characterization of geometric mean.
\end{proof}

 \begin{remark}  As $|X|\ge X$ for any self-adjoint $X$, we  find that (\ref{op-kantorovich5a}) is stronger than (\ref{op-kantorovich5b}). The reason to present  (\ref{op-kantorovich5a}) and  (\ref{op-kantorovich5b}) separately is that they are different. As $\Phi(A^{-1})$ and $\Phi(A)$ do not commute generally,  $\Phi(A^{-1})\Phi(A)+\Phi(A)\Phi(A^{-1})$ need not be positive in general.  \end{remark}

  \begin{remark}  Both  (\ref{op-kantorovich5a}) and  (\ref{op-kantorovich5b}) may be regarded as operator versions of Kantorovich inequality. We note that a similar version to  (\ref{op-kantorovich5b}) had been established earlier by Sun \cite{Sun91}, in the process of extending Bauer-Fike inequality on condition numbers. However, his line of proof is quite different from ours. Sun's operator version of Kantorovich inequality seems to be overlooked by the community, though it was proved almost at the same time as the appearance of   (\ref{op-kantorovich1}).  Interestingly, Sun's extension of Bauer-Fike inequality on condition numbers can also be proved using  (\ref{op-kantorovich1}); see \cite{Sun91} for more details. \end{remark}

\section{Comments}
When talking about the Kantorovich inequality, we cannot go without mentioning its cousin, the Wielandt inequality. The inequality of Wielandt \cite[p. 443]{HJ85} states that if $0<m\le A \le M$, and $x, y\in \mathcal {H}$ with $x\perp y$, then
 \begin{eqnarray}\label{e5} |\langle x, Ay\rangle|^2\le  \left(\frac{M-m}{M+m}\right)^2\langle x,Ax\rangle\langle y, Ay\rangle.
\end{eqnarray}

Let $\Phi$ be a unital positive linear map between $C^*$-algebras. We say that $\Phi$ is $2$-positive if whenever
the $2\times 2$ operator matrix  \[ \begin{bmatrix}A &  B  \\ B^*  & C\end{bmatrix}\] is positive, then so is  \[ \begin{bmatrix}\Phi(A) & \Phi(B) \\ \Phi(B^*)  & \Phi(C)\end{bmatrix}.\]

Operator version of  (\ref{e5}) was proved by  Bhatia \& Davis \cite{BD00} (independently by Wang \& Ip \cite{WI99}).

\begin{theorem}\label{thm3}\cite{BD00}  Let $0<m\le A\le M$ and let $X, Y$ be two
partial isometries on $\mathcal {H}$ whose final spaces are orthogonal to each other. Then for every $2$-positive linear map $\Phi$, \begin{eqnarray}\label{op-wielandt1}  \Phi(X^*AY)\Phi(Y^*AY)^{-1}\Phi(Y^*AX)\le \left(\frac{M-m}{M+m}\right)^2 \Phi(X^*AX).
\end{eqnarray}
  \end{theorem}

It is worthwhile to mention that Bhatia \& Davis's proof of Theorem \ref{thm3} reveals that the Kantorovich inequality and the Wielandt inequality are essentially equivalent.

The Schwarz inequality proved by Lieb and Ruskai \cite{LR74} says that for every $2$-positive linear map $\Phi$,
\begin{eqnarray}\label{schwarz} \Phi(T^*T)\ge \Phi(T^*S)\Phi(S^*S)^{-1}\Phi(S^*T)
\end{eqnarray}
Here if $\Phi(S^*S)$ is not invertible, $\Phi(S^*S)^{-1}$ is understood to be its generalized inverse.

With Proposition \ref{prop0}, (\ref{op-wielandt1}) and  (\ref{schwarz}) lead to the following corollary.
\begin{cor}  Under the same condition as in Theorem \ref{thm3}. Then for every $2$-positive linear map $\Phi$,
 \begin{eqnarray}\label{op-wielandt2}  \Big(\Phi(X^*AY)\Phi(Y^*AY)^{-1}\Phi(Y^*AX)\Big)\sharp \Phi(X^*AX)^{-1}\le \frac{M-m}{M+m}.
\end{eqnarray} Also, we have \begin{eqnarray}\label{schwarz2} \Big(\Phi(T^*S)\Phi(S^*S)^{-1}\Phi(S^*T)\Big)\sharp \Phi(T^*T)^{-1}\le I.
\end{eqnarray}
  \end{cor}

 \begin{proposition} There are examples that \begin{eqnarray}\label{schwarz-fail}\|\Phi(T^*S)\Phi(S^*S)^{-1}\Phi(S^*T)\Phi(T^*T)^{-1}\|\le 1
\end{eqnarray} fails, where $\Phi$ is some $2$-positive linear map.
  \end{proposition}
  \begin{proof} Note that (\ref{schwarz-fail}) is equivalent to \begin{eqnarray*} \Big(\Phi(T^*S)\Phi(S^*S)^{-1}\Phi(S^*T)\Big)^2\le \Phi(T^*T)^{2}.
\end{eqnarray*} In particular, \begin{eqnarray}\label{schwarz-faila}\left(XC^{-1}X^*\right)^2\le D^2,
\end{eqnarray} whenever $\begin{bmatrix}C & X^*  \\ X  &  D\end{bmatrix}$ is positive.   If $0< A\le B$, then $\begin{bmatrix}A^{-1} &  I  \\ I  &  B\end{bmatrix}$ is positive. Assume (\ref{schwarz-faila}) is valid, then it would imply
$A^2\le B^2$, which is absurd.   \end{proof}

  Despite the invalidity of (\ref{schwarz-fail}), there is some evidence that the following assertion could be true.  We have been unable to prove (or disprove) it.
   \begin{conj}\label{conjecture}  Under the same condition as in Theorem \ref{thm3}. Then for every $2$-positive linear map $\Phi$,
\begin{eqnarray}\label{op-wielandt3} \|\Phi(X^*AY)\Phi(Y^*AY)^{-1}\Phi(Y^*AX)\Phi(X^*AX)^{-1}\|\le \left(\frac{M-m}{M+m}\right)^2.
\end{eqnarray}
  \end{conj}

Following the line of  the proof of Theorem \ref{thm2}, we can present the following:

 \begin{proposition} Assume that Conjecture \ref{conjecture} is true. For every $2$-positive linear map $\Phi$, define \[\Gamma=\Phi(X^*AY)\Phi(Y^*AY)^{-1}\Phi(Y^*AX)\Phi(X^*AX)^{-1}.\] Then
\begin{eqnarray}\label{op-wielandt4a} \frac{1}{2}|\Gamma+\Gamma^*|\le \left(\frac{M-m}{M+m}\right)^2
\end{eqnarray} and \begin{eqnarray}\label{op-wielandt4b} \frac{1}{2}\Big(\Gamma+\Gamma^*\Big)\le \left(\frac{M-m}{M+m}\right)^2.
\end{eqnarray}
  \end{proposition}

\section{Acknowledgement}
The author thanks J. C. Bourin and M. S. Moslehian for helpful comments on an earlier version of the manuscript. Comments from the referee are also gratefully acknowledged.

\vskip 10pt

\noindent

Minghua Lin

 Department of Applied Mathematics,

University of Waterloo,

 Waterloo, ON, N2L 3G1, Canada.

mlin87@ymail.com

\end{document}